\documentclass[a4paper,10pt]{article}
\usepackage{enumerate}
\usepackage{color}
\usepackage[utf8]{inputenc} 
\usepackage[english]{babel}
\usepackage[T1]{fontenc}
\usepackage{graphicx}
\usepackage{amsfonts,amssymb,amsmath,latexsym,amsthm}
\usepackage{textcomp}
\usepackage[pdftex]{hyperref}

\title{Comparison of Volumes of Siegel Sets and Fundamental Domains  for $\mathrm{SL}_n (\mathbb{Z})$ }
\author{Gisele Teixeira Paula}
\date{}

\begin{document}
\maketitle
{\footnotesize Instituto Nacional de Matemática Pura e Aplicada. Est. Dona Castorina, 110. CEP 22460-320.	Rio de Janeiro-RJ. Brazil.}

{\footnotesize Correspondence to be sent to: e-mail: giseletp@impa.br}

\theoremstyle{definition}
\newtheorem{definition}{Definition}[section]

\newtheorem{proposition}{Proposition}[section]

\newtheorem{corollary}{Corollary}[section]
\newtheorem{lemma}{Lemma}[section]
\newtheorem{theorem}{Theorem}[section]

\newtheorem{exe}{Example}[section]
\newtheorem{rmk}{Remark}[section]

\newcommand{\R}{\mathbb{R}}
\newcommand{\Z}{\mathbb{Z}}
\newcommand{\s}{\mathbb{S}}
\newcommand{\G}{\mathrm{SL} _n (\mathbb{R})}
\newcommand{\Ga}{\mathrm{SL} _n (\mathbb{Z})}
\newcommand{\K}{\mathrm{SO} _n}
\newcommand{\q}{\Gamma\backslash G}
\newcommand{\Gaqz}{Q_{\mathbb{Z}} \backslash \Gamma}
\newcommand{\Gqz}{Q_{\mathbb{Z}} \backslash G}
\newcommand{\Si}{\Sigma _{t,\lambda}}
\newcommand{\I}{\mathcal{I}}
\newcommand{\In}{\left\{1,\ldots,n\right\}}
\newcommand{\g}{\mathfrak{g}}
\newcommand{\lgrec}{\mathfrak{l}}
\newcommand{\p}{\mathfrak{p}}
\newcommand{\lieg}{\mathfrak{sl}(n, \mathbb{R})}
\newcommand{\so}{\mathfrak{so}(n)}

	\maketitle
	
	\begin{abstract}
		The purpose of this paper is to calculate explicitly the volumes of Siegel sets which are coarse fundamental domains for the action of  $\Ga$ in $\mathrm{SL} _n (\mathbb{R})$, so that we can compare these volumes with those of the fundamental domains of $\Ga$ in $\mathrm{SL} _n (\mathbb{R})$, which are also computed here, for any $n\geq 2$. An important feature of this computation is that it requires keeping track of normalization constants of the Haar measures. We conclude that the ratio between volumes of fundamental domains and volumes of Siegel sets grows super-exponentially fast  as $n$ goes to infinity.  As a corollary, we obtained that this ratio gives a super-exponencial lower bound, depending only on $ n $, for the number of intersecting Siegel sets. We were also able to give an upper bound for this number, by applying some results on the heights of intersecting elements in $ \Ga $.
	\end{abstract}

	\section{Introduction}
	\label{intro}
	Siegel sets were first introduced in the study of quadratic forms by Siegel \cite{siegel2} in 1939, with some results following from previous works of Hermite and Korkine-Zolotarreff. In a fundamental paper \cite{borelharish}, Borel and Harish-Chandra have ge\-ne\-ra\-lised this notion and used Siegel domains to prove finiteness of covolumes of non-cocompact arithmetic subgroups. 
	
	The simple structure of  Siegel sets, compared to those of the actual fundamental domains makes them appealing for applications. For example, in his recent paper \cite{young}, R. Young exploited their properties to obtain new results in geometric group theory.  Still very little is known about the geometry of Siegel sets in general. In his book \cite{morris}, Morris describes algebraically examples of Siegel sets not only for $\mathrm{SL} _n (\mathbb{R})$, with $n\geq 2$ , but also in the case of any semisimple Lie group G with a given Iwasawa decomposition.
	
	In this paper we recall one of the main properties of Siegel sets -- the finiteness of their volumes. We evaluate these volumes explicitely in the basic case of Siegel sets for $\Ga$ in $\mathrm{SL} _n (\mathbb{R})$ for any $n\geq 2$. We then compare these volumes with the actual covolumes of $\Ga$. To this end, we have to deal with an essential difficulty related to the normalization of the Haar measure. For calculating the volumes of Siegel sets, the main difficulty is to find a nice way to describe the region of integration, which we solve with an appropriate change of coordinates. Most of the volume computations that followed Siegel's original approach were not careful about the normalization constants, just noting that they are computable and could be calculated from the proof. In Section \ref{domfund}, we follow Garret's notes on Siegel's method \cite{garret} to compute the volumes of the quotients $\Ga \backslash \mathrm{SL} _n (\mathbb{R})$ for $n\geq 3$ using induction and the volume of $\mathrm{SL}_2(\Z) \backslash\mathrm{SL}_2(\R)$, that is computed in \cite{garret}. Our main goal here is to keep a careful track of the normalization constants. The main tools we use are the Poisson Summation formula, the Iwasawa decomposition of $G$ and the choice of a good Haar measure normalization on each group. At the end of the section we discuss the relation between the normalization of the measure we used and the canonical normalization that comes from the metric associated to the Killing form on $\mathfrak{sl}_n(\R)$.

	By comparing the volumes of Siegel sets and the volumes of fundamental domains of $\Ga$, we conclude that somewhat surprinsingly the ratio between them grows super-exponentially  fast with $n$.
	
	As an application of the computations presented here, in Section \ref{morr} we show that given a Siegel set $\Sigma$ of $\Ga$, we have an explicit lower bound for the number of elements $\gamma \in \Ga$ such that $\gamma \Sigma$ intersects $\Sigma$. This bound is given by the ratio between $\mathrm{vol}(\Sigma)$ and $\mathrm{vol}(\mathrm{SL}_n(\Z) \backslash\mathrm{SL}_n(\R))$ -- see Corollary \ref{corol1}. We also give a proof that this result is consistent with a recent work of M. Orr \cite{martinorr}, which generalizes a previous result of P. Habegger and J. Pila \cite{habegger} on the height of such elements $\gamma$, motivated by the study of Shimura varieties and their unlikely intersections. More precisely, Orr's result gives, as a corollary, an upper bound for the number of intersecting Siegel sets while our work provides a lower bound for this number (see Corollary \ref{final}).

	It would be interesting to compute the volumes of Siegel sets in other cases, for example for the action of well known Bianchi groups $\Gamma_d = \mathrm{SL}_2(\mathcal{O}_d)$ on the hyperbolic three-dimensional space $\mathbb{H}^3$. In this case we should have to deal with another difficulty when describing Siegel sets, because of the fact that as $d$ grows the quotients $\Gamma_d \backslash \mathbb{H}^3$ have a growing number of cusps. It would be worth doing these computations in the future, and then comparing them to the results obtained in this paper.
	
	\section{The Iwasawa decomposition of \texorpdfstring{$\mathrm{SL} _n (\mathbb{R})$}.}
	\label{iwasawa}
	Let $n\geq 2$, $G=\mathrm{SL}_n(\mathbb{R})$ and  $\Gamma = \mathrm{SL}_n(\mathbb{Z})$. Consider the action of $\Gamma$ by left translations on $G$ and let
	
	$$K = \mathrm{SO}_n;$$
	$$A =\left\{\mbox{diag}(a_1,\ldots ,a_n); \displaystyle{ \prod_{i=1}^n{a_i} = 1} ; a_i > 0, \mbox{ for any } i=1,\ldots, n\right\};$$
	$$N =\left\{(n_{ij})_{i,j} \in G ; n_{ii}=1 \mbox{ and } n_{ij}= 0 \mbox{ for } i>j\right\}.$$
	
	\begin{lemma}
		The product map 
		$$ \Phi: K\times A \times N \longrightarrow G$$
		$$(k,a,n)\mapsto kan$$
		is a homeomorphism.
	\end{lemma}
	
	\begin{proof}
		We can construct an inverse map for $\Phi$ by using the Gram-Schmidt orthonormalization process.
		
		Take $g\in G$ and let $x_1, \ldots ,x_n$ be its columns. Then define inductively $y_1, \ldots ,y_n$ by 
		$$y_1 = \frac{x_1}{\left\|x_1\right\|};$$
		$$y_i = \frac{\widetilde{y}_i}{\left\|\widetilde{y}_i\right\|}, \mbox{ where } \widetilde{y}_i = x_i - \displaystyle{\sum_{l=1}^{i-1}{\left\langle x_i,y_l\right\rangle}y_l};  \mbox{ for } i = 2, \ldots, n.$$
		
		Let $e_1, \ldots ,e_n$ be the standard orthonormal basis of $\mathbb{R}^n$. Then there exists an unique $k\in \K$ such that $k(y_i) = e_i$, $\mbox{ for any } i = 1, \ldots n$.
		Therefore $$k(\widetilde{y}_i) = k(\left\|\widetilde{y}_i\right\| y_i) = \left\|\widetilde{y}_i\right\| k(y_i) = \left\|\widetilde{y}_i\right\|e_i,\mbox{ for any } i=1, \ldots , n.$$ 
		So there is a diagonal matrix $a = \mbox{diag}(\left\|\widetilde{y}_1\right\|, \ldots, \left\|\widetilde{y}_n\right\|)$, such that $$k(\widetilde{y}_i) =a(e_i), \mbox{ for any } i=1, \ldots , n.$$
		Also, it is easy to see that $y_i \in \left\langle x_1,\ldots , x_i\right\rangle, \mbox{ for any } i=1, \ldots , n$. Thus we have:
		$$g^{-1}\widetilde{y}_i= g^{-1}(x_i - \displaystyle{\sum_{l=1}^{i-1}{\left\langle x_i,y_l\right\rangle}y_l})  \in g^{-1}x_i + g^{-1}\left\langle x_1, \ldots , x_{i-1}\right\rangle$$
		$$\Rightarrow g^{-1}\widetilde{y}_i \in e_i + \left\langle e_1, \ldots , e_{i-1}\right\rangle.$$
		
		From this, we conclude that there exists $u \in N $ such that $g^{-1}\widetilde{y}_i= u e_i$, for every $i= 1, \ldots, n$. Therefore, 
		$$u^{-1}g^{-1}\widetilde{y}_i= e_i = a^{-1}k(\widetilde{y}_i), \mbox{ for any } i \Rightarrow u^{-1}g^{-1} = a^{-1}k \Rightarrow g=k^{-1}a u^{-1}.$$
		It is easy to see now that $\mathrm{det}(a) =1$, so  $a \in A$ and thus we can define a continuous inverse map $g\in G \mapsto (k^{-1}, a, u^{-1} ) \in K\times A \times N$.
	\end{proof}
	
	The previous lemma gives us the Iwasawa decomposition $G=KAN$ of $\mathrm{SL} _n (\mathbb{R})$. Note that $K\cap A = K\cap N = A \cap N = \{I\}$ and that for this Iwasawa decomposition, $ AN = NA $ and $ K(AN) = (AN)K $ (see \cite{morris}, page 148). 
	
	\section{Haar measure on \texorpdfstring{$\mathrm{SL} _n (\mathbb{R})$}.} 
	\label{haar}
	
	Given a locally compact Hausdorff topological group $G$,  a left invariant Haar measure on $G$ is, by definition, a regular Borel measure $\mu$ on $G$ such that for all $g \in G$ and all Borel sets $E \subset G$ we have $\mu (gE) = \mu (E)$. It is well known that every connected Lie group admits such a Haar measure. Moreover, it is unique up to scalar multiples. We can define analogously right-invariant Haar measures. See \cite{venka} for more results about Haar measures on Lie groups.
	
	Since $G = \mathrm{SL} _n (\mathbb{R})$ is unimodular, i.e. the left and right invariant Haar measures coincide, and $dg$ is invariant under left translation by elements of $K$ and under right translation by elements of $AN$, we get that the Haar measure of $G$ in $k, v, a$ coordinates is given by the product measure $dg=dk\, du\, da$, where $da$, $du$ and $dk$ are the Haar measures on $A$, $N$ and $K$, repectively. This means that for every compactly supported and continuous function $f$ on $G$, we have $$\int_G{f(g) dg} = \int_K \int_{A} \int_{N} {f(kau) du \, da \, dk}.$$
	
	It can be proved by induction on $n$ that the Haar measure on $N$ is given by $du= \displaystyle \prod_{i<j}{du_{ij}}$. It is usually convenient to change the order of integration on the variables $a$ and $u$, and to this end we can change the coordinates from $u$ to $v=aua^{-1}$. Then $v$ is also an upper triangular unipotent matrix of the form 
	$$v = Id + \displaystyle \sum_{i<j\leq n} \frac{a_i}{a_j}u_{ij}E_{ij}.$$
	It is easily seen that $dv= \displaystyle \prod_{i<j}dv_{ij} = \displaystyle \prod_{i<j}\frac{a_i}{a_j}du_{ij}$.  This gives us 
	$$\int_G{f(g) dg} = \int_K \int_{A} \int_{N} {f(kva) dv \, da \, dk}= \int_K \int_{N} \int_{A} {f(kau)\displaystyle \prod_{i<j}\frac{a_i}{a_j} da \, du \, dk}.$$
	Also for convenience, we change coordinates from $ au $ to $ k^{-1}auk $ in the last integral. This has Jacobian equal to 1 (for each $ k\in K $), so we get:
	
	$$ \int_G{f(g) dg} =  \int_{N} \int_{A}\int_K {f(auk)\displaystyle \prod_{i<j}\frac{a_i}{a_j} dk \, da \, du}.$$
	
	In this work, we will consider the Haar measure in $K$ to be the following: it is easy to see that the isotropy group of $e_n =(0, \ldots, 0, 1)$ by the action of $\K$ in $\s^{n-1}$ is isomorphic to $\mathrm{SO}_{n-1}$. Then $\s^{n-1} \cong \mathrm{SO}_{n-1} \backslash \K$. We have that the natural map $\pi : \K \rightarrow \s ^{n-1}$ is  a Riemannian submersion if we rescale it by a factor of $\frac{1}{\sqrt{2}}$. Thus 
	
	$$\mathrm{vol}(\K) = 2^{\frac{1}{2}(n-1)} \mathrm{vol}(\mathbb{S}^{n-1})\cdot \mathrm{vol}(\mathrm{SO}_{n-1}).$$
	
	By using induction and the fact that $\mathrm{vol}(\s ^{n-1}) = \frac{2\pi ^{\frac{n}{2}}}{\Gamma(\frac{n}{2})}$, we obtain 
	$$ \mathrm{vol}(\K) = 2^{\frac{1}{4}n(n-1)} \mathrm{vol}(\mathbb{S}^{n-1})\cdot \mathrm{vol}(\mathbb{S}^{n-2}) \ldots \mathrm{vol}(\mathbb{S}^{1}) = 2^{(n-1)(\frac{n}{4}+1)} \prod^n_{i=2}{\frac{\pi^\frac{i}{2}}{\Gamma (\frac{i}{2})}}.$$
	
	It remains to define a Haar measure on $A$. We claim that $da = \displaystyle \prod _{i=1}^{n-1} {\frac{da_i}{a_i}}$ is such a measure. Indeed, let $\phi: A \rightarrow \R^{n-1}$ be the map 
	$$a = \left(
	\begin{array}{ccccc}
	a_{1} & 0 &  \ldots & 0 & 0 \\
	0 & a_{2} & \cdots & 0 & 0 \\
	\vdots & \vdots  & \ddots & \vdots & \vdots \\
	0 & 0 & \cdots  & a_{n-1} & 0\\
	0 & 0 & \cdots   & 0 & \displaystyle \prod _{i=1}^{n-1}{a_i^{-1}}\\
	\end{array}
	\right) \mapsto (t_1 , \ldots , t_{n-1}) = (\log a_1, \ldots ,\log a_{n-1}).$$
	As $\phi$ is a group isomorphism and Haar measure is preserved by isomorphisms we get that 
	$da = \displaystyle \prod _{i=1}^{n-1} dt_i = \displaystyle \prod _{i=1}^{n-1} {\frac{da_i}{a_i}}$ is a Haar measure on $A$.

	\section{Siegel Sets for \texorpdfstring{$\mathrm{SL} _n (\mathbb{R})$}.}
	\label{Siegelsets}	
	
	\begin{definition}
		Let $\Gamma$ be some group acting properly discontinuously on a topological space $X$. We call $\mathcal{F} \subset X$ a coarse fundamental domain for $\Gamma$ if:
		
		\begin{itemize}
			\item $\Gamma \mathcal{F} = X$;
			\item $\left\{\gamma \in \Gamma ; \gamma \mathcal{F} \cap \mathcal{F} \neq \emptyset \right\}$ is finite.
		\end{itemize}
		
	\end{definition}
	
	\begin{definition}
		A Siegel set in $\mathrm{SL} _n (\mathbb{R})$ is a set $\Si$ of the form 
		$$\Sigma _{t,\lambda} = A_tN_\lambda K,$$
		where $t,\lambda$ are positive real numbers,
		$$A_t = \left\{a \in A ; \frac{a_i}{a_{i+1}} \leq t , \mbox{ for any } i =1, \ldots , n\right\}$$
		and
		$$N_\lambda = \left\{u \in N | \left|u_{ij}\right|\leq \lambda , \mbox{ for any } i, j =1, \ldots , n\right\}.$$
	\end{definition}
	
	For certain parameters $t, \lambda$ the Siegel sets $\Si$ are coarse fundamental domains for $\Ga$. Another important pro\-per\-ty is that they have finite volume. Siegel sets can be also defined in a more general way for lattices in other semisimple Lie groups, as it can be seen in Chapter 19 of Morris \cite{morris}. In many cases, a finite union of copies of Siegel sets glue together to form coarse fundamental domains for general lattices.
	
	In this section we compute the volumes of the Siegel sets in $\mathrm{SL} _n (\mathbb{R})$. We will use the Haar measure on $G$ given in Section \ref{haar} in $v,a,k$ coordinates. 
	\begin{theorem}
		\begin{equation*}
		\mathrm{vol}(\Si) = \frac{1}{2} \mathrm{vol}(\mathrm{SO}_n) (2\lambda)^{\frac{n(n-1)}{2}}\frac{t^{\frac{n(n^2-1)}{6}}}{((n-1)!)^2} .
		\label{eq1}
		\end{equation*}
	\end{theorem}
	
	\begin{proof}
		$$\mathrm{vol}(\Si) = \int_{\left|u_{ij}\right| \leq \lambda }\int_{\frac{a_i}{a_{i+1}} \leq t}  \int_{K} \prod_{i<j}{\frac{a_i}{a_j}}
		dk \prod_{i=1}^{n-1} {\frac{da_i}{a_i}} \prod_{1\leq i <j \leq n}{du_{ij}}.$$
		
		$$= \mathrm{vol}(K) (2\lambda)^{\frac{n(n-1)}{2}} \int_{\frac{a_i}{a_{i+1}} \leq t}{\prod_{i<j}{\frac{a_i}{a_j}}\prod_{i=1}^{n-1}{\frac{da_i}{a_i}}}.$$
		
		To compute the integral over  $a_1, \ldots, a_n$ (with the condition $\displaystyle \prod_{i=1}^{n}{a_i} = 1$), we change variables from $a_1, \ldots, a_n$ to the variables 
		$$b_i = \frac{a_i}{a_{i+1}}, \mbox{ for any } i = 1, \ldots , n-1.$$
		
		By elementary computation, we get 
		$$\displaystyle \prod_{i<j}{\frac{a_i}{a_j}} = \prod_{i=1}^{n-1}{b_i ^{i(n-i)}}.$$
		Moreover, as $a_i = b_i a_{i+1}$, the Jacobian of the change of coordinates from $a_i$ to $b_i$ is $\frac{1}{2a_1}$. The integral then becomes
		$$\int_{b_i \leq t} {\displaystyle \prod_{i=1}^{n-1}{b_i ^{i(n-i)}} \frac{1}{b_1 a_2b_2 a_3 \ldots b_{n-1} a_n}  \frac{1}{2a_1}\prod_{i\leq n-1}{db_i}}$$
		$$=\frac{1}{2} \int_{b_i \leq t} {\prod_{i=1}^{n-1}{b_i ^{[i(n-i)-1]}} \prod_{i\leq n-1}{db_i}} = \frac{1}{2} \prod_{i=1}^{n-1}{\frac{t^{ni-i^2}}{(ni - i^2)}} = \frac{1}{2} \frac{t^{\frac{n(n^2-1)}{6}}}{((n-1)!)^2}.$$
		
		Thus we get to 
		\begin{equation}
		\mathrm{vol}(\Si) = \frac{1}{2} \mathrm{vol}(\mathrm{SO}_n) (2\lambda)^{\frac{n(n-1)}{2}}\frac{t^{\frac{n(n^2-1)}{6}}}{((n-1)!)^2} .
		\label{eq1}
		\end{equation}
		
	\end{proof}

	Borel proves in \cite{borel} the following theorem:
	
	\begin{theorem}
		For $t\geq \frac{2}{\sqrt{3}}$ and $\lambda \geq \frac{1}{2}$, one has $\Sigma _{t,\lambda} \Gamma = G$. Moreover, $ \Sigma _{t,\lambda}  $ is a coarse fundamental domain for $ \Gamma $ in $ G $.
	\end{theorem}
	
	\begin{corollary}
		The quotient $ \q $ has finite volume, which satisfies 
		\begin{equation}
		\mathrm{vol}(\q) \prec e^{cn^3}, \mbox{ as } n \rightarrow \infty,
		\end{equation}
		for some positive constant $ c $.
	\end{corollary}
	
	\begin{proof}
		It is clear that $\mathrm{vol}(\q) < \infty$, since $\Sigma _{t,\lambda}$ has finite volume and it contains a fundamental domain for $\Ga$ if $t\geq \frac{2}{\sqrt{3}}$ and $\lambda \geq \frac{1}{2}$. Thus $\mathrm{vol}(\q) \leq \mathrm{vol}(\Sigma_{t,\lambda})$ for these values of $ t $ and $ \lambda $.
		
		By taking $\lambda=\frac{1}{2}$ and $t= \frac{2}{\sqrt{3}}$ in formula \eqref{eq1}, we get that 
		$$\mathrm{vol}(\Sigma_{{\scriptscriptstyle{\frac{2}{\sqrt{3}}, \frac{1}{2}}}}) = 2^{(n-1)(\frac{n}{4}+1)-1} \Big( \prod ^n_{i=2}{\frac{\pi^\frac{i}{2}}{\Gamma (\frac{i}{2})}} \Big) \frac{(\frac{2}{\sqrt{3}})^{\frac{n(n^2-1)}{6}}}{((n-1)!)^2}$$
		$$=\frac{ 2^{\frac{2n^3 +3n^2+ 7n -24}{12}} \pi^{\frac{n^2+n-2}{2}}}{3^{\frac{n(n^2-1)}{12}} ((n-1)!)^2 \displaystyle \prod^n_{i=2}{\Gamma (\frac{i}{2})}}$$
		Using Stirling's formula, this volume is easily seen to grow assymptotically like $e^{cn^3}$, for some positive constant $c$ and this finishes the proof. 
		
	\end{proof}

	On the other hand, as we will see in the next section, $\mathrm{vol}(\Ga \backslash \mathrm{SL} _n (\mathbb{R}))$ computed with respect to the same normalization of the Haar measure goes to zero as $n$ grows.

	\section{Volume of \texorpdfstring{$\Ga \backslash \mathrm{SL} _n (\mathbb{R})$}.} 
	
	\label{domfund}
	
	It is a well-known fact that $\mathrm{vol}(\Ga \backslash \mathrm{SL} _n (\mathbb{R}))$ is finite. Our goal is to calculate it, with respect to the same normalization of the Haar measure used in the previous section. The whole computation follows the original approach of Siegel \cite{siegel}, but we have to be careful with the normalization constants. We use Poisson summation, induction and the previously known fact that $\mathrm{vol}(\mathrm{SL}_2(\Z) \backslash \mathrm{SL}_2(\R)) = \sqrt{2}\zeta (2)$, which can be proved in a similar way (see \cite{garret}, being careful with respect to the different normalization of $\mathrm{vol}(\mathrm{SO} _2)$ we are considering).

	We will state first the Poisson Summation Formula, which will play a fundamental role in the computations, and for which the reader can refer to \cite{psf}.
	
	Given a lattice $\Lambda$ in $\R^n$, we define $\left|\Lambda \right|$ to be the covolume of $\Lambda$, i.e. the volume of $\R^n/ \Lambda$ and the dual lattice of $\Lambda$ by 
	$$\Lambda^* = \left\{y \in \R^n; \left\langle x,y\right\rangle \in \Z \mbox{ for any } x \in \Lambda \right\}.$$
	
	\begin{theorem}[Poisson Summation Formula]
		\label{psf}
		Given any lattice $\Lambda$ in $\R^n$, a vector $w \in \R^n$ and an adimissible function $f:\R^n \rightarrow \R$ in $\mathcal{L}^1$, we have 
		$$\displaystyle \sum_{x \in \Lambda}{f(x+w)} = \frac{1}{\left|\Lambda\right|}\sum_{t \in \Lambda^*}{e^{-2\pi i \left\langle w,t\right\rangle}\hat{f}(t)},$$
	\end{theorem}
	
	Here, $\hat{f}(t) = \int_{\R^n}{f(x) e^{2\pi i\left\langle x,t\right\rangle} dx}$ is the Fourrier transform of $f$ and admissibilty of $f$ means that there exist constants
	$\epsilon, \delta > 0$ such that $\left|f(x)\right|$ and $\left|\hat{f}(x)\right|$ are bounded above by  $\epsilon(1+\left|x\right|)^{-n-\delta}$.

	Let then $f \in \mathcal{L}^1$ be an admissible function on $\R ^n$. We can ask $ f $ to be a $ C^{\infty} $ function with compact support. We then define  $F:G\mapsto\R$ by
	$$F(g) = \displaystyle \sum_{v \in \Z ^n}{f(vg)}.$$
	Here we are considering the multiplication of line-vectors $v \in \R ^n$ by elements of $ G $ by the right. Clearly, $F$ is left $\Gamma$-invariant, as $\Z ^n \Gamma = \Z^n$ under the action of $\Ga$ on $\R ^n$ by right multiplication of  line vectors by the inverse elements of $ \Ga $.

	Consider $\int_{\q}{F(g)dg}$. We will use this integral to calculate $\mathrm{vol}(\q)$. 
	
	Let
	$$Q = \mathrm{stab}_G(e)  = \left\{ \left( \begin{array}{cc}
	h & v  \\
	0 & 1  \\
	\end{array} \right); h \in \mathrm{SL}_{n-1}(\R), v \in \R ^{n-1}\right\},$$ 
	where $e=(0,\ldots, 0,1) \in \R^{n}$, and write $Q_{\Z} = Q \cap \Gamma $. Using linear algebra over $\Z$, note that 
	$$\Z ^{n} - \{0\} = \displaystyle\bigcup_{\ell >0} \displaystyle\bigcup_{\gamma \in \Gaqz } \ell e \gamma, $$
	where $\ell$ runs over positive integers.

	Then we can write
	$$\displaystyle \int_{\q}{F(g)dg} = \int_{\q}{f(0)dg} + {\int_{\q}{\sum_{\ell >0}\sum_{\gamma \in \Gaqz }f(\ell e \gamma g)dg}}$$
	$$ = \mathrm{vol}(\q)f(0) + \sum_{\ell >0} \int_{Q_{\Z} \backslash G}{f(\ell eg)dg}.$$
	
	For the second equality note that a fundamental domain for $Q_{\Z} $ in $G$ is the union of images of a fundamental domain for $\Gamma$ in $G$ by representatives of classes in $\Gaqz $. In addition, the Schwartz condition on $f$ ensures that the integral over $Q_{\Z} \backslash G$ is finite. Indeed, in his article \cite{siegel45}, Siegel proves that the function $ F(g) $ is integrable over a fundamental domain for $ \q $. On pages 344-345 of [loc. cit.] we can see that the integrals over $ \mathbb{Q}_\Z \backslash G$ are also convergent (for any fixed $ l\in \mathbb{N} $). We observe that although he uses different decomposition of $ G $ and normalization of Haar measures, this does not change the finiteness of the integrals.
	
	Write 
	$$P = \left\{ \left( \begin{array}{cc}
	h & *  \\
	0 & \frac{1}{\mbox{det}(h)} \\
	\end{array} \right); h \in \mathrm{GL}_{n-1}(\R) , \mbox{det}(h) > 0 \mbox{ and } * \in \R^{n-1} \right\};$$
	$$N' = \left\{ \left( \begin{array}{cc}
	I_{n-1} & v  \\
	0 & 1  \\
	\end{array} \right); v \in \R^{n-1}  \right\}, N'_{\Z} = N' \cap \Gamma.$$
	$$M = \left\{ \left( \begin{array}{cc}
	h & 0  \\
	0 & 1  \\
	\end{array} \right); h \in \mathrm{SL}_{n-1}( \R) \right\}, M_{\Z} = M \cap \Gamma;$$
	$$A' = \left\{ \left( \begin{array}{cc}
	t^{{\scriptscriptstyle\frac{1}{n-1}}}I_{n-1} & 0  \\
	0 & t^{-1}  \\
	\end{array} \right); t > 0 \right\};$$
	
	Note that $P = N'MA' \supset NA$, $Q= N'M$ and $G = N'MA'K$. However this time we have that $N'MA'$ intersects $K$ non-trivially, i.e. this is not an Iwasawa decomposition. The product $N'MA'K$ projects on $G$ with fiber $\mathrm{SO} (n-1)$. Therefore we get the following
	
	\begin{lemma}
		For every left $G$-invariant function $\Phi$, we have
		$$\displaystyle \int_{\Gqz}{\Phi (g)dg} = \frac{1}{\mathrm{vol}(\mathrm{SO}_{n-1})} \int_{Q_{\Z} \backslash ( N'MA'K)}{\Phi (n'ma'k)dn'\, dm\,da' \, dk}$$
		where $dg$ is the Haar measure in $ G $ coming from its Iwasawa decomposition (as in Section \ref{Siegelsets}).
	\end{lemma}
	
	Here $dn'$, $dm$, $da'$ and $dk$ are the left Haar measures on $N'$, $M$, $A'$ and $K$, respectively. We see that  $ dn' = \displaystyle \prod_{i=1}^{n-1}v_i $ and that $ M $ is isomorphic to $ \mathrm{SL}_{n-1}(\R) $ and thus $ dm $ will appear as the measure of this group. This  allows us to use induction in the calculations. On the other hand,  $ A' $ is isomorphic to $ \R_{>0} $ via the  isomorphism $$\left( \begin{array}{cc}
	t^{{\scriptscriptstyle\frac{1}{n-1}}}I_{n-1} & 0  \\
	0 & t^{-1}  \\
	\end{array} \right) \in A' \mapsto t \in \R_{>0}.$$
	Thus we have $ da'= \frac{dt}{t} $  where $dt$ is the usual measure in $ \R $.
	
	Again it will be convenient to change the order of integration, by letting the variable $a' \in A'$ to be the last one. This will give us $d(a' q a'^{-1}) = t^n dq$, for $q = n'm \in N'M$. 
	Indeed, for $ a' = \left( \begin{array}{cc}
	t^{{\scriptscriptstyle\frac{1}{n-1}}}I_{n-1} & 0  \\
	0 & t^{-1}  \\
	\end{array} \right) \in A'$ and $ q = \left( \begin{array}{cc}
	h & v  \\
	0 & 1  \\
	\end{array} \right) \in Q $, we have $ a'q (a')^{-1}  = \left( \begin{array}{cc}
	h & t^{\frac{n}{n-1}}v  \\
	0 & 1  \\
	\end{array} \right)$, and thus the $ M $-contribuction to the measure doesn't change, but the $ N' $-contribution is multiplied by $ (t^{\frac{n}{n-1}})^{n-1} = t^n $ and we get to $d(a' q a'^{-1}) = t^n dq$ as stated.
	
	Then, if we require $f$  to be $K$-invariant, the integral $\int_{\q} {F(g)dg}$ becomes equal to
	
	$$\mathrm{vol}(\q)f(0) + \frac{1}{\mathrm{vol}(\mathrm{SO}_{n-1})}\sum_{\ell >0} \int_{Q_{\Z} \backslash ( N'MA' K)}{f(\ell en'ma'k) dn' \, dm\, da' \, dk}$$
	$$= \mathrm{vol}(\q)f(0) + \frac{\mathrm{vol}(Q_{\Z} \backslash K )}{\mathrm{vol}(\mathrm{SO}_{n-1})}\sum_{\ell >0} \int_{Q_{\Z} \backslash ( N'M)}\int_{A'} {f(\ell en'ma')\displaystyle t^n da' \,  dn' \, dm }.$$
	
	We have $K \cap Q_{\Z} = \mathrm{SO}_{n-1}(\Z)$. Noting that 
	$$\s^{n-1} \cong \mathrm{SO}_{n-1} \backslash \mathrm{SO}_n \cong \frac{\mathrm{SO}_{n-1}(\Z) \backslash \mathrm{SO}_n}{\mathrm{SO}_{n-1}(\Z) \backslash \mathrm{SO}_{n-1}},$$
	we get  
	$$\mathrm{vol}(\s^{n-1}) = \mathrm{vol}(\mathrm{SO}_{n-1} \backslash \mathrm{SO}_n) =  \frac{\mathrm{vol}(\mathrm{SO}_{n-1}(\Z) \backslash \mathrm{SO}_n)}{\mathrm{vol}(\mathrm{SO}_{n-1}(\Z) \backslash \mathrm{SO}_{n-1})}.$$
	
	As $\mathrm{SO}_{n}(\Z)$ acts properly and freely in $\mathrm{SO}_n$, for any $n \in \mathbb{N}$ we have that $$\mathrm{SO}_n \longrightarrow \mathrm{SO}_{n}(\Z) \backslash \mathrm{SO}_n$$ is a finite covering with $\# \mathrm{SO}_{n}(\Z)$ sheets, which gives us 
	$$\mathrm{vol}(\mathrm{SO}_n) = \#(\mathrm{SO}_{n}(\Z)) \mathrm{vol}(\mathrm{SO}_{n}(\Z) \backslash \mathrm{SO}_n).$$ 
	Altogether, we obtain:
	$$\mathrm{vol}(Q_{\Z} \backslash K )  = \mathrm{vol}(\mathrm{SO}_{n-1}(\Z) \backslash \mathrm{SO}_n) = \frac{\mathrm{vol}(\s^{n-1}) \mathrm{vol}(\mathrm{SO}_{n-1})}{\# (\mathrm{SO}_{n-1}(\Z))}.$$
	
	As the integrand is invariant under $N'M$ ($en'm = e$, for any $n' \in N'$  and $m \in M$) and the volume of $N'_{\Z} \backslash N'$ is $1$, this implies
	
	$$\mathrm{vol}(\q)f(0) + \frac{\mathrm{vol}(Q_{\Z} \backslash K )}{\mathrm{vol}(\mathrm{SO}_{n-1})}\sum_{\ell >0} \int_{Q_{\Z} \backslash ( N'M)}\int_{A'} {f(\ell en'ma')\displaystyle t^n da'\, dn' \, dm}$$
	$$=\mathrm{vol}(\q)f(0) + \frac{\mathrm{vol}(Q_{\Z} \backslash K )}{\mathrm{vol}(\mathrm{SO}_{n-1})} \mathrm{vol}(\mathrm{SL}_{n-1}( \Z) \backslash \mathrm{SL}_{n-1}( \R)) \sum_{\ell >0} \int_{A'} \! \! \! \! \! {f(\ell ea')t^n da'}$$
	$$= \mathrm{vol}(\q)f(0) + \frac{\mathrm{vol}(\s^{n-1})}{\# (\mathrm{SO}_{n-1}(\Z))} \mathrm{vol}(\mathrm{SL}_{n-1}( \Z) \backslash \mathrm{SL}_{n-1}( \R)) \sum_{\ell >0} \int_{A'} f(\ell ea')t^n da'$$
	
	By replacing $a' \in A'$ by $t \in \R _{>0}$ and using the description of $ da' $, we get to 
	$$\mathrm{vol}(\q)f(0) + \frac{\mathrm{vol}(\s^{n-1}) }{\# (\mathrm{SO}_{n-1}(\Z))} \mathrm{vol}(\mathrm{SL}_{n-1}( \Z) \backslash \mathrm{SL}_{n-1}( \R)) \sum_{\ell >0} \int_{0}^{\infty} f(\ell et)t^n \frac{dt}{t}.$$
	
	By replacing $t$ by $\frac{t}{\ell}$, we obtain
	$$\mathrm{vol}(\q)f(0) + \frac{\mathrm{vol}(\s^{n-1}) }{\# (\mathrm{SO}_{n-1}(\Z))} \mathrm{vol}(\mathrm{SL}_{n-1}( \Z) \backslash \mathrm{SL}_{n-1}( \R)) \sum_{\ell >0}\frac{1}{\ell ^n} \int_{0}^{\infty} f(et)t^n \frac{dt}{t}.$$
	
	By using polar coordinates in $\R^n = \{ (v, t), v \in \mathbb{S}^{n-1}, t \in \R _{>0}\}$, we get
	$$\mathrm{vol}(\s ^{n-1})\int_{0}^{\infty} f(et)t^n \frac{dt}{t} = \int_{\s^{n-1}}\int_{0}^{\infty} f(v,t)t^{n-1} dtdv = \int_{\R^{n}} f(x)dx = \hat{f}(0).$$
	
	Thus what we get until now is the following
	
	\begin{proposition} The initial integral becomes 
		$$\displaystyle \int_{\q}{F(g)dg} = 	\mathrm{vol}(\q)f(0) + \frac{\mathrm{vol}(\mathrm{SL}_{n-1}( \Z) \backslash \mathrm{SL}_{n-1}( \R))}{\# (\mathrm{SO}_{n-1}(\Z))}  \zeta(n) \hat{f}(0),$$	
		where $\zeta(n) = \displaystyle \sum_{l\in\Z}{\frac{1}{l^n}}$ is the Riemman zeta function.
	\end{proposition}
	
	\begin{corollary} 
		The previous result allows us to compute explicitely the value of $ \mathrm{vol}(\q) $:
		$$\mathrm{vol}(\q) = \frac{ \mathrm{vol}(\mathrm{SL}_{n-1}( \Z) \backslash \mathrm{SL}_{n-1}(\R))}{\# (\mathrm{SO}_{n-1}(\Z))}  \zeta(n) = \sqrt{2} \displaystyle \prod_{i=2}^{n}\zeta(i) \prod_{i=1}^{n-1}{\frac{1}{\# (\mathrm{SO}_{i}( \Z))}}  .$$
	\end{corollary}
	
	\begin{proof}
		For every $g \in G$, we are going to apply the Poisson summation formula to the lattice $\Lambda = \left\{vg;v \in \Z^n \right\}$ in $\R^n$, the vector $w=0$ and the initial function $f$. Note that $\Lambda^* = \left\{vg^*; v \in \Z^n \right\}$, where $g^* = ^{\top}\! \! g^{-1}$. Then we get
		$$F(g) = \displaystyle \sum_{v \in \Z ^n}{f(vg)} = \displaystyle \sum_{v \in \Z ^n}{\hat{f}(vg^*)} = \hat{F}(g^*), \mbox{ for any } g \in G.$$
		
		The automorphism $g \mapsto g^*$ preserves the measure on $G$ and stabilizes $\Gamma$, so we can do an analogous computation with the roles of $f$ and $\hat{f}$ reversed. Since $\hat{\hat{f}}(0) = f(0)$ and $\int_{\Gamma \backslash G}{F(g)dg} = \int_{\Gamma \backslash G}{\hat{F}(g)dg}$, we obtain 
		$$\mathrm{vol}(\q)f(0) + \frac{ \mathrm{vol}(\mathrm{SL}_{n-1}( \Z) \backslash \mathrm{SL}_{n-1}( \R))}{\# (\mathrm{SO}_{n-1}(\Z))}  \zeta(n) \hat{f}(0) = \int_{\q}\!\!\!{F(g)dg}$$
		$$=\int_{\Gamma \backslash G}\!\!\!{\hat{F}(g)dg} = \mathrm{vol}(\q)\hat{f}(0) + \frac{\mathrm{vol}(\mathrm{SL}_{n-1}( \Z) \backslash \mathrm{SL}_{n-1}( \R))}{\# (\mathrm{SO}_{n-1}(\Z))}  \zeta(n) f(0).$$
		
		By asking additionally that $f$ is  such that $f(0) \neq \hat{f}(0)$ and using indution on $n$, we get to the desired result.
	\end{proof}

	We observe that for every $i\in \mathbb{N}$, $\# (\mathrm{SO}_{i}( \Z)) = 2^{i-1}i!$. Indeed, the group $\mathrm{SO}_{i}( \Z)$ consists of monomial matrices whose nonzero entries are equal to $\pm 1$ and which have determinant equal to $1$. The first condition gives us $2^i i!$ matrices. Now if we look at the surjective group homomorphism 
	$$\mathrm{det}: B=\{\mbox{monomial matrices with nonzero entries} \in \{\pm 1\}\} \rightarrow \{\pm 1 \},$$
	we get $B/Ker(\mathrm{det}) \cong \{\pm 1 \}$, which implies $$\# (\mathrm{SO}_{i}( \Z)) = \#(Ker (\mathrm{det})) = \frac{\#(B)}{2} = \frac{2^{i}i!}{2} = 2^{i-1}i! .$$
	
	Thus we have proved the following
	
	\begin{theorem}
		The explicit volume of $ \q $, by considering the Haar measures described in Section \ref{haar} is given by 
		\begin{equation*}
		\mathrm{vol}(\Ga \backslash \mathrm{SL} _n (\mathbb{R})) =\sqrt{2} \displaystyle \prod_{i=2}^{n}\zeta(i) \displaystyle \prod_{i=1}^{n-1}{\frac{1}{2^{i-1}i!}} = \frac{\displaystyle \prod_{i=2}^{n}\zeta(i)}{2^{{\scriptscriptstyle\frac{n^2-3n+1}{2}}}  \displaystyle \prod_{i=2}^{n} i!} .
		\label{eq2}
		\end{equation*}
	\end{theorem}

	It is not difficult to see that this function goes to zero like $e^{-c'n^2}$ as $n$ grows, where $c'$ is a positive constant. It has a completely different behaviour from the volume growth of Siegel Sets described by formula \eqref{eq1}. What we can conclude directly from all this is that although the geometry of a Siegel set is simpler than that of the actual fundamental domain for a lattice, their volumes can differ dramatically as $n$ grows. Thus we should be careful if we want to replace fundamental domains of any lattice by simpler structures such as Siegel sets, due to the possibility that some of their relevant geometric features, e.g. volume, may have different behavior to that of fundamental domains. 
	
	As a consequence of Sections \ref{Siegelsets} and \ref{domfund}, we obtain:
	
	\begin{corollary}
		The ratio between volumes of the minimal Siegel sets $\Sigma = \Sigma_{{\scriptscriptstyle{\frac{1}{2}, \frac{2}{\sqrt{3}}}}}$ for $ \Ga $ and the actual fundamental domains  for these groups in $ \G $ is given by 
		$$C(n) = \frac{\mathrm{vol}(\Sigma)}{\mathrm{vol}(\q)} = \frac{2^{{\scriptscriptstyle\frac{2n^3+ 9n^2+25n-30}{12}}}\pi^{{\scriptscriptstyle\frac{n^2+n-2}{4}}} \displaystyle \prod_{i=1}^{n-1}{i!}}{3^{{\scriptscriptstyle\frac{n^3-n}{12}}}((n-1)!)^2 \displaystyle \prod_{i=2}^{n}{\Gamma(\frac{i}{2})}  \displaystyle \prod_{i=2}^n{\zeta(i)}}. $$
		
		Moreover, $ C(n) \sim e^{\tilde{c}n^3} $ for some constant $\tilde{c}$ that does not depend on $n$.	
	\end{corollary}

	A natural question arising here is the following: ``How is our normalization of the Haar measure related to the canonical normalization defined by using the Killing form on $\mathfrak{sl}_n(\R)$?''
	
	To answer to this question we can compare our formula with a result of Harder \cite{harder}, who computed the volume of $\Ga \backslash X$, where $X$ is the symmetric space $\mathrm{SL} _n (\mathbb{R})/\K$. In order to do this comparison, note that by equation \eqref{eq2} we have
	\begin{equation}
	\mathrm{vol}(\Ga \backslash X) =\frac{\mathrm{vol}(\q)}{\mathrm{vol}(\K)} = \frac{\sqrt{2} \displaystyle \prod_{i=1}^{n-1}{\frac{1}{2^{i-1}i!}}\displaystyle \prod_{i=2}^{n}{\zeta(i)}}{2^{(n-1)(\frac{n}{4}+1)} \displaystyle  \prod^n_{i=2}{\frac{\pi^\frac{i}{2}}{\Gamma (\frac{i}{2})}}}.
	\label{eq3}
	\end{equation}
	
	By Harder's formula, we obtain that this volume in the canonical nor\-ma\-li\-za\-tion is given by 
	\begin{equation}
	\mathrm{vol}_1 (\Ga \backslash X) = \frac{\displaystyle \prod_{i=1}^{n-1}{i!}\displaystyle \prod_{i=2}^{n}{\zeta(i)}}{(2\pi)^{{\scriptscriptstyle\frac{n(n+3)}{2}}} 2^{\tau}n!}, 
	\label{eq4}
	\end{equation}
	where $\tau = n$ if $n$ is odd and $\tau = n-1$ if $n$ is even.
	
	We see that these volumes differ by a factor given by
	\begin{equation}
	C_1(n) = \frac{\mathrm{vol}_1 (\Ga \backslash X)}{\mathrm{vol}(\Ga \backslash X)}= \frac{2^{{\scriptscriptstyle\frac{n^2-5n-2}{4} - \tau }}   \Bigl(\displaystyle \prod_{i=1}^{n-1}{i!}\Bigl)^2  }{ n!\pi^{{\scriptscriptstyle\frac{n^2+5n+2}{4}}}  \displaystyle  \prod^n_{i=2}{\Gamma \Bigl(\frac{i}{2}\Bigl)}},
	\end{equation}
	
	where $\tau = n$ if $n$ is odd and $\tau = n-1$ if $n$ is even. We note that again by using Stirling's formulas, we obtain that $ C_1(n) $ grows assymptotically with $ n $ like $ e^{\kappa n^2} $, for some positive constant $ \kappa $.
	
	The same renormalization can be applied to \eqref{eq1} in order to obtain the volumes of Siegel sets in the symmetric spaces with respect to the standard normalization of the measure.

	\section{Bounding the number of intersecting domains}
	\label{morr}
	
	Another relevant consequence of this work is the following corollary:
	\begin{corollary}
		\label{corol1}
		Let  $N$  be the cardinality of the set $\mathcal{I} :=\left\{\gamma \in \Gamma ; \gamma \Sigma \cap \Sigma \neq \emptyset\right\}$, where  $\Sigma = \Sigma_{{\scriptscriptstyle{\frac{1}{2}, \frac{2}{\sqrt{3}}}}}$ . Then $N\geq C(n) = \frac{\mathrm{vol}(\Sigma)}{\mathrm{vol}(\q)}$. 
		
	\end{corollary}
	
	\begin{proof}
		As $\Sigma$ is a Siegel set, it must contain a fundamental domain $\mathcal{F}$ for $\Gamma$.  We affirm that $\Sigma \subset \underset{{\scriptscriptstyle \gamma \in \mathcal{I}}}{\bigcup} \gamma\mathcal{F}$.
		
		Indeed, given $x \in \Sigma$, if $x\in \mathcal{F}$, there is nothing to prove. If $x \notin \mathcal{F}$, as the images of  $\mathcal{F}$ tesselate $\mathrm{SL} _n (\mathbb{R})$ we must have $x \in \gamma \mathcal{F}$, for some $Id \neq \gamma \in \Gamma$. As $\gamma \mathcal{F} \subset \gamma \Sigma$, we obtain $x \in \gamma \Sigma \cap \Sigma$, and thus $\gamma \in \mathcal{I}$. Therefore the inclusion above is true.
		
		From this we obtain $N \mathrm{vol}(\q) = N \mathrm{vol}(\mathcal{F}) \geq \mathrm{vol}(\Sigma)$ and thus $N \geq C(n)$, as stated.
	\end{proof}

	In his recent work \cite{martinorr}, Martin Orr shows in a more general setting that given a reductive algebraic group $G$ defined over $\mathbb{Q}$, a general Siegel set $\Sigma \subset G(\R)$ for some arithmetic subgroup $\Gamma \subset G(\mathbb{Q})$,  and $\theta \in G(\mathbb{Q})$, there exists an upper bound for the height of elements $\gamma\in \Gamma$ such that $\theta \Sigma \cap \gamma \Sigma \neq \emptyset$. The height of an element is defined by:
	$$H(\gamma) = \displaystyle \max_{1\leq i,j\leq n} H(\gamma_{ij}),$$
	where given a rational number $a/b$, $H(a/b)$ is defined as the maximum of the absolute values of $a$ and $b$. 
	Orr shows that, given any element $\gamma$ of the set 
	$$\Sigma_{N,D} := \Sigma \Sigma^{-1}\cap \left\{\gamma \in G(\mathbb{Q}), \mathrm{det} \gamma \leq N\mbox{ and the denominators of  } \gamma \mbox{ are } \leq D\right\},$$
	there exists some constant $C_1$, depending on the group $G$, on the Siegel set $\Sigma$ and on the way the group $G$ is embedded in some $GL_n(\R)$, such that 
	$$H(\gamma) \leq C_1N^nD^{n^2},$$
	where $N = \left|\mathrm{det} \gamma\right|$ and $D$ is the maximum of the denominators of entries of $\gamma$. Note that for $\Gamma = \Ga$, the set  $\mathcal{I}$ defined above is contained in $\Sigma_{N,D}$.
	
	In this section we are going to compare this result with ours, i.e., to see what happens in the case when $G= \mathrm{SL} _n (\mathbb{R})$ and $\Gamma = \Ga$. Note that in this case, for any $\gamma \in \Gamma$, we have $N= \left|\mathrm{det} \gamma\right| = 1$ and also $D = 1$ because the entries of $\gamma$ are all integers. Thus Orr's result gives us, for this case, 
	$$H(\gamma)\leq C_1(n).$$
	By the definition, the height of an element $\gamma \in \Ga$ is equal to  $\left|\gamma\right|_{max}$. Therefore, his result turns to 
	$$\left|\gamma \right|_{max} \leq C_1(n), \mbox{ for any } \gamma \in \Sigma \Sigma^{-1}.$$
	
	By Example $1.6$ on page $5$ of  \cite{sarnack}, the set $\left\{\gamma \in \Ga;\left\|\gamma \right\| \leq C_1(n) \right\}$ has cardinality of assymptotic order $c_nC_1(n)^{(n^2-n)}$, with $c_n \rightarrow 0$ as $n\rightarrow\infty$. Thus if we assume that $n$ is sufficiently large, we can suppose that $c_n < \epsilon$ for some $\epsilon>0$ fixed. Therefore, we have 
	$$\left|\left\{\gamma \in \Ga;\left\|\gamma \right\| \leq C_1(n) \right\}\right| \prec C_1(n)^{(n^2-n)},$$
	where the notation $f(n) \prec g(n)$ used above means that there exists a positive constant $C$ such that for sufficiently big $n$, we have $f(n)\leq Cg(n)$.
	
	Note that the result in \cite{sarnack} is  proved for the Euclidean norm $\left\|.\right\|$ in $M_{n\times n}$ and we know that $\left\|\gamma\right\| \leq n\left|\gamma\right|_{max}$. Thus 
	$$\left|\left\{\gamma \in \Ga;\left|\gamma\right|_{max} \leq C_1(n) \right\}\right| \prec (n C_1(n))^{(n^2-n)}.$$
	
	We are going to show that $$C_1(n)\leq  e^{\frac{n^2-n}{2} ln(n)}.$$
	From this we obtain that $\left|\mathcal{I}\right| \prec e^{\frac{n^4}{2}ln(n)}$. 
	Hence we have:
	
	\begin{corollary}
		\label{final}
		For $\Gamma = \Ga $ in $\mathrm{SL} _n (\mathbb{R})$ and $\mathcal{I}$ defined above, there exist constants $c_1, c_2 >0$ such that 
		$$e^{c_1 n^3} \leq \left|\mathcal{I}\right| \leq e^{c_2 n^4 ln(n)}.$$
	\end{corollary}

	In order to obtain the second inequality we adapt the proofs in \cite{martinorr} for the $\mathrm{SL} _n (\mathbb{R})$ case, with the difference that we give explicit values for the constants. 
	
	\begin{definition}
		Let $\gamma \in \I$. From this element, we can define: 
		
		\begin{itemize}
			\item A partition of $\In$ (with respect to $\gamma$) is a list of disjoint subintervals of $\In$, which we call components, whose union is all of $\In$ and such that:
			\begin{itemize}
				\item $\gamma $ is block upper triangular with respect to the chosen partition;
				\item $\gamma $ is not block upper triangular with respect to any other finer partition of $\In$;
			\end{itemize}
			\item A leading entry of $\gamma$ is a pair $(i,j) \in \In^2$ such that $\gamma_{ij}$ is the leftmost non-zero entry of the $i$-th row of $\gamma$.
		\end{itemize}
	\end{definition}
	For a concrete description of what are the possible partitions in the $\mathrm{GL}_3$ case see Section 3.2 of \cite{martinorr}.
	
	We will make use of the following lemma whose proof can be found in \cite{martinorr}:
	
	\begin{lemma}
		\label{lema1}
		If $i, j$ are in the same component, then there exists a sequence of indices $i_1, \ldots, i_s$ such that $i_1 =i, i_s = j$ and
		$$(*) \mbox{ For every } p \leq s-1, \mbox{ either } i_p \leq i_{p+1} \mbox{ or } (i_p, i_p+1) \mbox{ is a leading entry.}$$
	\end{lemma}
	
	In the proof of the following lemmas for the $\mathrm{GL}_n$ case,  Martin Orr uses the notation $A\ll B$ meaning that there exists a constant $C$, depending on $n$,  such that $\left|A\right| \leq C\left|B\right|$. Our point here is to compute such constants so that we can make explicit the value of $C_1(n)$. 
	
	\begin{lemma}
		\label{lema2}
		If $(i,j)$ is a leading entry of $\gamma$, then $\alpha_j \leq \sqrt{n} \beta_i $. 
	\end{lemma}
	
	\begin{proof}
		For any $\gamma \in \Sigma \Sigma^{-1}$, we can write $\gamma= \nu \beta \kappa \alpha^{-1}\mu^{-1}$, with $\kappa \in \K$, $\nu, \mu \in N_{\frac{1}{2}}$ and $\alpha, \beta \in A_{\frac{2}{\sqrt{3}}}$. This gives us the equation $\gamma \mu \alpha = \nu \beta \kappa$. We will compare the lengths of the $i$-th rows on each side of this equation. 
		
		As $\kappa \in \K$, multiplying by $\kappa$ on the right does not change the length of each row. If we expand out lengths we obtain
		$$\displaystyle \sum_{p=1}^{n}{\Big(\sum_{q=1}^{n}{\gamma_{iq} \mu_{qp}}\Big)\alpha_{p}^2} = \sum_{p=1}^{n}{\nu_{ip}^2\beta_p^{2}}.$$
		
		As $\nu$ is upper triangular, the non-zero terms on the right hand side of the last equation must have $p \geq i$. By the definition of $A_t$, for all $p \geq i$ we have
		$$\beta_p \leq \frac{1}{t^{(p-i)}}\beta_i \leq \beta_i,$$
		where in the second inequality we used that $t=\frac{2}{\sqrt{3}}$ and $p\geq i $ imply  $ \frac{1}{t^{(p-i)}} \leq 1$. Since $\nu \in N_{\frac{1}{2}}$,  $\left|\nu_{ip}\right| \leq 1$ for any $i,p$. Alltogether, 
		$$ \sum_{p=1}^{n}{\nu_{ip}^2\beta_p^{2}} \leq \left|\sum_{p=1}^{n}{\nu_{ip}^2\beta_p^{2}}\right| \leq  \sum_{p\geq i}^{n}{\beta_p^{2}} \leq (n-i)\beta_i^2 \leq n \beta_i^2.$$
		
		On the other hand, by looking at the left hand side of the equation, we obtain:
		$$\Big(\sum_{q=1}^{n}{\gamma_{iq} \mu_{qj}}\Big)\alpha_{j}^2 \leq \displaystyle \sum_{p=1}^{n}{\Big(\sum_{q=1}^{n}{\gamma_{iq} \mu_{qp}}\Big)\alpha_{p}^2}.$$
		As $(i,j)$ is a leading entry, we can only have $\gamma_{iq} \neq 0$ if $q \geq j$. But as $\mu$ is upper triangular, $\mu_{qj} \neq 0$ implies $q\leq j$. Thus the only non-zero term in the first sum is the one for $q=j$ and then we get
		$$\Big(\sum_{q=1}^{n}{\gamma_{iq} \mu_{qj}}\Big)\alpha_{j}^2  = \gamma_{ij}^2\mu_{jj}^2\alpha_j^2 = \gamma_{ij}^2 \alpha_j^2.$$
		Note that $\gamma_{ij}\neq 0$ and that as $\gamma$ has integer entries, we must have $\left|\gamma_{ij}\right| \geq 1$, which implies $\gamma_{ij}^2\geq 1$.
		
		Altogether, we obtain $$\alpha_j^2 \leq  \alpha_j^2\gamma_{ij}^2 \leq n\beta_i^2 \Rightarrow \alpha_j \leq \sqrt{n}\beta_i,$$
		from what we conclude the proof.
	\end{proof}
	
	\begin{lemma}
		\label{lema3}
		For all $k\in \In$,  $\alpha_k \leq \sqrt{n} \beta_k $. 
	\end{lemma}
	
	\begin{proof}
		We affirm that there must exist a leading entry $(i,j)$ such that $j\leq k \leq i$. To prove this notice that as $\gamma$ is invertible, there must exist $i\geq k$ such that the $i$-th row of $\gamma$ contains a non-zero entry in the $k$-th column or to its left (otherwise the leftmost k columns of $\gamma$ would have rank less than k). Choose j so that $\gamma_{ij}$ is the leading entry of $\gamma$ in the $i$-th line and it will satisfy $j\leq k$ as claimed.
		
		By Lemma \ref{lema2} and by the definition of $A_t$ we obtain
		$$\alpha_k \leq \frac{1}{t^{(k-j)}}\alpha_j \leq \sqrt{n} \beta_i \leq \sqrt{n}\frac{1}{t^{(i-k)}}\beta_k \leq \sqrt{n} \beta_k.$$
	\end{proof}

	\begin{lemma}
		\label{lema4}
		For all $j\in \In$,  $\beta_j  \leq (\sqrt{n})^{n-1}  \alpha_j $. 
	\end{lemma}
	
	\begin{proof}
		As $\alpha$ and $\beta$ are diagonal with positive real entries, we have (by using Lemma \ref{lema3} in the inequality)
		$$\beta_j \mathrm{det}(\alpha) = \beta_j \displaystyle \prod_{k=1}^n{\alpha_k} \leq \beta_j \alpha_j (\sqrt{n})^{n-1}\prod_{k\neq j}{\beta_k} = (\sqrt{n})^{n-1} \alpha_j \mathrm{det}(\beta).$$
		But as $\mathrm{det}(\beta) = \mathrm{det}(\alpha)= 1$,
		$$\beta_j \leq (\sqrt{n})^{n-1} \alpha_j $$
		and the lemma is proved.
	\end{proof}
	
	\begin{lemma}
		\label{lema5}
		If $i$ and $j$ are in the same component, $\beta_j \leq (\sqrt{n})^{n^2-n}\alpha_{i}$.
	\end{lemma}
	
	\begin{proof}
		We can apply Lemma \ref{lema1} to obtain a sequence $i_1=i, \ldots, i_s=j$ such that for any $p\in \left\{1,\ldots,s\right\},$ we have either $i_p\leq i_{p+1}$ or $(i_p,i_{p+1})$ is a leading entry. We take this subsequence as the smallest possible. 
		
		If $i_p\leq i_{p+1}$ then as $\alpha \in A_t$ and $(\sqrt{n})^{n}\geq 1$, we get 
		$$\frac{\alpha_{i_p}}{\alpha_{i_{p+1}}}\geq t^{i_{p+1}-i_p}\geq 1 \Rightarrow \alpha_{i_{p+1}}\leq \alpha_{i_{p}} \leq (\sqrt{n})^{n}\alpha_{i_{p}}.$$
		On the other hand if $(i_p,i_{p+1})$ is a leading entry then by Lemmas \ref{lema2} and \ref{lema4} we have 
		$$\alpha_{i_{p+1}}\leq \sqrt{n}\beta_{i_p} \leq (\sqrt{n})^{n}\alpha_{i_p}.$$
		If we apply the last inequality successively we get to 
		$$\alpha_j = \alpha_{i_s} \leq (\sqrt{n})^{n(s-1)}\alpha_{i}.$$
		Now we just apply Lemma \ref{lema4} and notice that $s\leq n$ to obtain
		$$\beta_j\leq (\sqrt{n}))^{n-1}\alpha_{j}\leq (\sqrt{n})^{n^2-1}\alpha_{i}.$$
	\end{proof}

	We write
	$$Q = \left\{g \in G; g \mbox{ is block upper triangular according to the components of }\gamma \right\};$$
	$$L = \left\{g \in G; g \mbox{ is block diagonal according to the components of }\gamma \right\}.$$
	
	We affirm that $\kappa\in L$. Indeed, as the matrices $\gamma, \mu, \alpha, \beta$ and $\nu$ are in Q by the construction, we also have $\kappa \in Q$. On the other hand, if a matrix is block upper triangular and is also orthogonal, then it is block diagonal. Thus $\kappa \in L$.
	
	\begin{lemma}
		\label{lema6}
		If $i, j \in \In$, then $\left|\gamma_{ij} \right| \leq C_1(n) = n^{\frac{n^2-n}{2}}$.
	\end{lemma}
	
	\begin{proof}
		Write $\gamma = \nu\beta\kappa\alpha^{-1}\mu^{-1}$. Because $\alpha, \beta$ are diagonal, the $pq$-th entry of $\beta\kappa\alpha^{-1}$ is $\beta_p\kappa_{pq}\alpha_q^{-1}$. 
		
		If $p$ and $q$ are not in the same component, as $\kappa \in L$, we get that $\kappa_{pq} = 0$. On the other hand, if they are in the same component, then by Lemma \ref{lema5} 
		$$\beta_p\kappa_{pq}\alpha_q^{-1} \leq \kappa_{pq} (\sqrt{n})^{n^2-1}.$$
		By the definition of $\K$,  $\left|\kappa\right|_{max}\leq 1$ for every $\kappa\in \K$. Therefore 
		$$\beta_p\kappa_{pq}\alpha_q^{-1} \leq (\sqrt{n})^{n^2-1}.$$
		
		As we have $\mu, \nu \in N_{\frac{1}{2}}$, we have $\left|\mu\right|_{\infty}, \left|\nu\right|_{\infty} \leq 1.$ Altogether, we obtain 
		$$\left|\gamma_{ij} \right| \leq (\sqrt{n})^{n^2-1}.$$
	\end{proof}
	
	Therefore we conclude the proof that $H(\gamma) \leq C_1(n)$, where $$C_1(n) = (\sqrt{n})^{n^2-1} = e^{\frac{n^2-1}{2} ln(n)}$$ and this finishes the proof of Corollary \ref{final}.

	
	%
	%
	
	\section*{Acknowledgements}
		I would like to thank Professor Mikhail Belolipetsky for several suggestions on the development of this paper and also on the text. I also thank Paul Garret and Martin Orr for their very helpful works and for always answering my emails with good suggestions, and Cayo D\'oria for helping me to understand better some topics. Finally, I also thank the refferee for carefully reading the paper and for giving suggestions that improved the presentation of the results. 
	
	

\end{document}